\documentclass[12pt,reqno]{amsart}
\usepackage{amsfonts,amsmath,amscd,amssymb,amsthm,graphicx}
\usepackage{enumerate,hyperref,perpage,url}

\usepackage[margin=2.2cm, a4paper]{geometry}

\theoremstyle{plain}
\newtheorem{thm}{\textbf{Theorem}}[section]

\newtheorem{prop}[thm]{\textbf{Proposition}}

\newtheorem{remark}{\textbf{Remark}}
\newtheorem{defn}{\textbf{Definition}}

\newtheorem{example}{\textbf{Example}}

\numberwithin{equation}{section}

\MakePerPage{footnote} \allowdisplaybreaks 

\newcommand{\norm}[1]{\left|\!\left|{#1}\right|\!\right|}

\newcommand\Id{\mathrm{Id}}

\newcommand\R{\mathbb{R}}

\title[$L^{p}$ estimates for joint quasimodes]{ $L^{p}$ estimates for joint quasimodes of semiclassical pseudodifferential operators}

\author{Melissa Tacy}

\email{mtacy@maths.otago.ac.nz}

\address{Department of Mathematics and Statistics, University of Otago}

\begin{document}

  \begin{abstract}
We develop a set of $L^{p}$ estimates for functions $u$ that are a joint quasimodes (approximate eigenfunctions) of $r$ semiclassical pseudodifferential operators $p_{1}(x,hD),\dots,p_{r}(x,hD)$. This work extends Sarnak \cite{S06} and Marshall's \cite{M16} work on symmetric space to cover a more general class of manifolds/operators. \end{abstract}
\maketitle

Let $(M,g)$ be a compact, boundaryless Riemannian manifold of dimension $n$. It is well established that there is a countably infinite set of eigenfunctions
\begin{equation}\Delta u_{j}=-\lambda_{j}^{2}u_{j}\quad \lambda_{j}\to\infty\label{eigeneq}\end{equation}
which can be normalised to produce an orthonormal basis for $L^{2}(M)$. An important question arising in harmonic analysis is to quantify the degree to which  eigenfunctions can be spatially concentrated. One way to measure this concentration is to compare the $L^{p}$ norm $u_{j}$ to its $L^{2}$ norm. In 1988 Sogge \cite{S88} obtained a set of estimates
$$\norm{u_{j}}_{L^{p}}\leq C{}\lambda_{j}^{\delta(n,p)}\norm{u_{j}}_{L^{2}}$$
where $\delta(n,p)$ is given by the piecewise linear function
$$\delta(n,p)=\begin{cases}
\frac{n-1}{2}-\frac{n}{p}&\frac{2(n+1)}{n-1}\leq{}p\leq\infty\\
\frac{n-1}{4}-\frac{n-1}{2p}&2\leq{}p\leq{}\frac{2(n+1)}{n-1}.\end{cases}$$
In fact the $L^{\infty}$ estimate is a consequence of the local Weyl law (as in \cite{L52} and \cite{H68}). However interpolation between the $L^{\infty}$ bound and the trivial $L^{2}$ bound does not produce sharp estimates for any of the intermediate $L^{p}$. In comparison Sogge's estimates are known to be sharp on the sphere. The high $p$ estimates (that is $p\geq{}2(n+1)/(n-1)$) are saturated by zonal harmonics and the low $p$ estimates (that is $p\leq{}2(n+1)/(n-1)$) by highest weight harmonics. The same $L^{p}$ bounds have be shown to hold for approximate eigenfunctions of semiclassical pseudodifferential operators with Laplace-like conditions on their symbol \cite{koch} and there is a rich literature of related results considering $L^{p}$ estimates on lower dimension subsets of $M$  (see for example \cite{BGT},\cite{Hu},\cite{HTacy},\cite{tacy09},\cite{chen14}).

In his letter to Morawetz \cite{S06}, Sarnak asks about potential improvements for joint eigenfunctions of the form
$$u=\phi_{1}(x_{1})\cdots\phi_{r}(x_{r})$$
where each $\phi_{i}$ is an eigenfunction of a differential (or pseudodifferential) operator $P_{i}$ with $P_{1}=\Delta$. He notes that on $S^{2}$ the invariance of zonal harmonics under rotation around the north pole prevents any improvement to the $L^{\infty}$ estimate where the second operator is the generator of rotations about the North/South axis,
$$P_{1}=\Delta\quad{}P_{2}=\partial_{\varphi}.$$
However, under the assumption that $M$ is a rank $r$ symmetric space, he shows that there is an improvement in the $L^{\infty}$ norm,
$$\norm{u_{j}}_{L^{\infty}}\lesssim{}\lambda_{j}^{\frac{n-r}{2}}\norm{u_{j}}_{L^{2}}.$$
This result is extended by Marshall \cite{M16} to a full set of $L^{p}$ estimates, 
$$\frac{\norm{u_{j}}_{L^{p}}}{\norm{u_{j}}_{L^{2}}}\lesssim\begin{cases}
\lambda_{j}^{r\delta(n/r,p)}&p\neq{}\frac{2(n+r)}{n-r}\\
(\log(\lambda_{j}))^{1/2}\lambda_{j}^{r\delta(n/r,p)}&p=\frac{2(n+r)}{n-r}.\end{cases}$$
In the other direction Toth and Zelditch \cite{TZ1}, \cite{TZ2} and \cite{TZ} study the behaviour of joint eigenfunctions in the completely integrable setting. In this case for any $(M,g)$ non-flat there are  sequences of joint eigenfunctions with
$$\norm{u_{j}}_{L^{\infty}}\geq{}C\lambda_{j}^{\frac{1}{4}-\epsilon}\norm{u}_{L^2}$$
compared to the symmetric case where $n=r$, then
$$\norm{u}_{L^{\infty}}\leq{}C\norm{u}_{L^{2}}.$$

In this paper we address the general problem of $L^{p}$ estimates for joint eigenfunctions. In particular we consider $u$ a joint solution (or approximate solution) to $r$ semiclassical pseudodifferential equations
$$p_{i}(x,hD)u=0\quad i=1,\dots,r\leq{}n$$
that obey a joint curvature condition. If  $p_{1}(x,hD)=-h^{2}\Delta-1$ (or $p_{1}(x,hD)$ is sufficiently Laplace-like) the curvature condition holds automatically. Here we set the semiclassical parameter $h=\lambda^{-1}$ so that eigenfunctions of $\Delta$ are solutions of $p_{1}(x,hD)u=0$. We use the left quantisation 
$$p_{j}(x,hD)u=\frac{1}{(2\pi h)^{n}}\int e^{\frac{i}{h}\langle x-y,\xi\rangle}p_{j}(x,\xi)u(y)dyd\xi$$
to associate a symbol $p_{j}(x,\xi)$ with an  operator $p_{j}(x,hD)$. It is necessary to place admissibility  conditions on the $p_{i}(x,\xi)$ (discussed in Section \ref{sec:conditions}) to exclude such cases as Sarnak's counter-example on $S^{2}$. The main theorem of this paper, Theorem \ref{thm:main}, gives a full set of $L^{p}$ estimates for strong joint quasimodes (see Definition \ref{defn:strongqm} for the definition of a strong joint quasimode). Theorem \ref{thm:n=r} tells us that if $n=r$ we can in fact get uniform bounds for $\norm{u}_{L^{p}}$ for any $p$.

\begin{thm}\label{thm:main}
Let $r<n$. Suppose $u$ is a semiclassically localised,  strong joint $O_{L^{2}}(h)$ quasimode for a set of semiclassical pseudodifferential operators $p_{1}(x,hD),\dots,p_{r}(x,hD)$ where the symbols $p_{j}(x,\xi)$ obey the following admissibility conditions
\begin{enumerate}
\item For each $x_{0}$ and $j$ the set $\{\xi\mid p_{j}(x_{0},\xi)=0\}$ is a smooth hypersurface in $T^{\star}_{x_{0}}M$.
\item If $\nu_{j}(x,\xi)$ is the normal to the hypersurface $\{\xi\mid p_{j}(x,\xi)=0\}$, then $\nu_{1},\dots,\nu_{r}$ are linearly independent. 
\item There is some $j$ such that for all $x_{0}$, $\gamma S_{j}\gamma^{\star}$ is non-degenerate. Here $S_{j}$ is the shape operator associated with $\{\xi\mid p_{j}(x_{0},\xi)=0\}$ and $\gamma$ is the projection onto the tangent space of $\cap_{j}\{\xi\mid p_{j}(x_{0},\xi)=0\}$. 

\end{enumerate}
Then
$$\norm{u}_{L^{p}}\lesssim h^{-\delta(n,p,r)}\norm{u}_{L^{2}},$$
$$\delta(n,p,r)=\begin{cases}
\frac{n-r}{2}-\frac{n-r+1}{p}&\frac{2(n-r+2)}{n-r}\leq{}p\leq\infty\\
\frac{n-r}{4}-\frac{n-r}{2p}&2\leq{}p\leq{}\frac{2(n-r+2)}{n-r}.\end{cases}$$

\end{thm}

\begin{remark}
The final assumption, which tells us about the curvature of $\cap_{j}\{\xi\mid p_{j}(x_{0},\xi)=0\}$, is necessary to produce the estimates for $2<p<\infty$. The $p=\infty$ estimate is true if only the first and second conditions on the $p_{j}(x,\xi)$ hold. It is this $p=\infty$ case that Sarnak is concerned with in \cite{S06} and connects with the sub-convex bounds considered by number theorists for example in \cite{RW}. In the case without curvature the best intermediate estimates are those given by interpolation between the $L^\infty$ estimate and the trivial $L^{2}$ estimate.
\end{remark}

\begin{figure}[h!]
\includegraphics[scale=0.5]{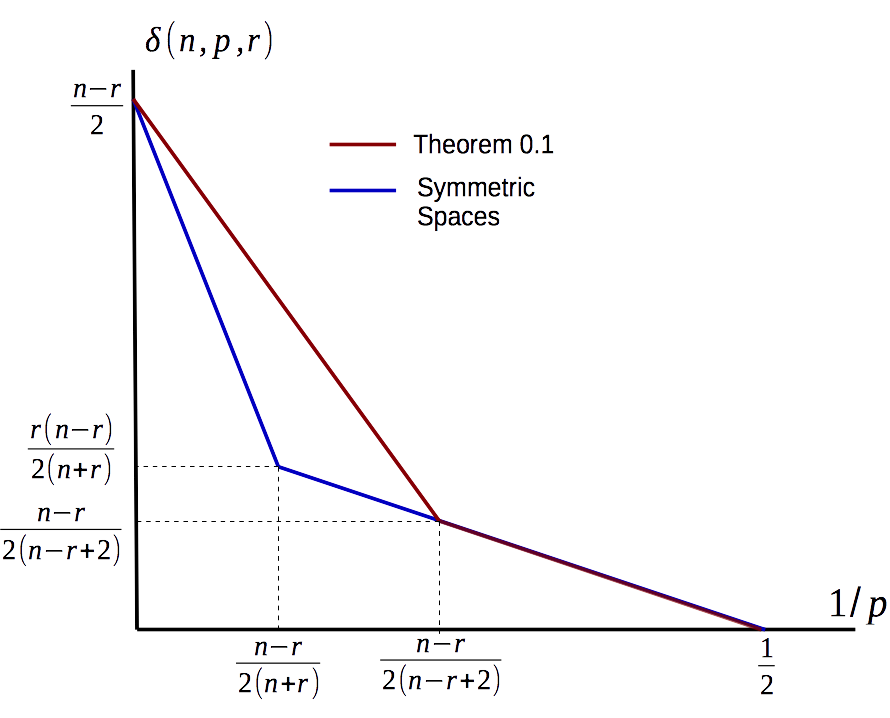}
\caption{The comparison between the results of Theorem \ref{thm:main} and those on symmetric spaces}
\label{fig:comp}
\end{figure}

\begin{remark}
Figure \ref{fig:comp} compares the results of Theorem \ref{thm:main} with Marshall and Sarnak's results on symmetric spaces. They agree for $p=\infty$ and for all $2\leq p\leq{}\frac{2(n-r+2)}{n-r}$ however in the range $\frac{2(n-r+2)}{n-r}< p<\infty$ symmetric spaces enjoy better estimates.
\end{remark}

\begin{remark}
An immediate question is whether the results of Theorem \ref{thm:main} could be improved to achieve the same results as on symmetric spaces. It is however relatively easy to construct an example that shows that Theorem \ref{thm:main} is sharp. Consider the flat model case $p_{1}(x,\xi)=|\xi|^{2}-1$ and $p_{i}(x,\xi)=\xi_{i}$ for $i=2,\dots,r$. Joint solutions to the equations $hD_{x_{i}}u=0, i=2,\dots,r$ take the form
$$u(x)=u(x_{1},x_{r+1},\dots,x_{n})$$
and therefore if $p_{1}(x,hD)u$ is an order $h$ quasimode Koch-Tataru-Zworski \cite{koch} tells us that it must satisfy Sogge's growth estimates for dimension $n-r+1$.  This then tells us that the results of Theorem \ref{thm:main} cannot be improved without further assumptions.
\end{remark}

In the case when $n=r$ Theorem \ref{thm:main} can be strengthened to say that quasimodes are uniformly bounded so long as the set of $p_{i}(x,\xi)$ obey the first two conditions (that is the curvature condition is unnecessary in this case).

\begin{thm}\label{thm:n=r}
Suppose $u$ is a semiclassically localised,  strong joint $O_{L^{2}}(h)$ quasimode for a set of semiclassical pseudodifferential operators $p_{1}(x,hD),\dots,p_{n}(x,hD)$ where the symbols $p_{j}(x,\xi)$ obey the following admissibility conditions
\begin{enumerate}
\item For each $x_{0}$ and $j$ the set $\{\xi\mid p_{j}(x_{0},\xi)=0\}$ is a smooth hypersurface in $T^{\star}_{x_{0}}M$.
\item If $\nu_{j}(x,\xi)$ is the normal to the hypersurface $\{\xi\mid p_{j}(x,\xi)=0\}$, then $\nu_{1},\dots,\nu_{r}$ are linearly independent. 
\end{enumerate}
Then
$$\norm{u}_{L^{p}}\lesssim \norm{u}_{L^{2}}$$
for any $2\leq{}p\leq{}\infty$.

\end{thm}

\begin{remark}
In both Theorem \ref{thm:main} and \ref{thm:n=r} the conditions (1)-(3) ((1)-(2) respectively) only need hold on a neighbourhood of the support of $\chi(x,\xi)$, where $\chi(x,hD)$ is the localisation operator as in Definition \ref{defn:localised}.
\end{remark}
Throughout this paper we will understand $f\lesssim g$ to mean
$$f\leq Cg$$
where $C$ is a constant that may depend on the manifold $M$ and the functions $p_{i}(x,\xi)$ but remains independent of the eigenvalue $\lambda$ (or the semiclassical parameter $h$).

\section{Quasimodes and semiclassical analysis}\label{sec:semiclassical}
We wish to study functions $u$ so that they approximately solve the equation
$$(-h^{2}\Delta-1)u=0$$
or some similar semiclassical equation. By working in coordinate charts and associating each patch with a patch on $\R^{n}$, we can write the operator $-h^{2}\Delta-1$ as a semiclassical quantisation of a symbol $p(x,\xi)$ which has principal symbol $\sigma_{p}(x,\xi)=|\xi|_{g}^{2}-1$. Here we use the left quantisation
\begin{equation}(-h^{2}\Delta-1)u=p(x,hD)u=\frac{1}{(2\pi h)^{n}}\int e^{\frac{i}{h}\langle x-y,\xi\rangle}p(x,\xi)u(y)dyd\xi.\label{pdef}\end{equation}
Since we must localise to make sense of \eqref{pdef} it is reasonable to only consider those $u$ which are semiclassically localised in phase space.

\begin{defn}\label{defn:localised}
We say that $u$ is semiclassically localised if there is a smooth, compactly supported function $\chi:T^{\star}M\to\R$ so that
$$u=\chi(x,hD)u+O(h^{\infty}).$$
\end{defn}

Throughout this paper we will use a some key standard results from semiclassical analysis. For the readers convenience we record the results here  and direct them to \cite{Zworski12} for the proofs and further discussion.

\begin{prop}[Composition of semiclassical $\Psi$DOs]\label{prop:com}
Let $p(x,hD)$, $q(x,hD)$ be left-quantised semiclassical pseudodifferential operators with symbols $p(x,\xi)$ and $q(x,\xi)$ respectively. The the symbol of $p(x,hD)\circ{}q(x,hD)$ is given by
\begin{multline}p(x,\xi)\# q(x,\xi)=e^{ih\langle{}D_{\xi},D_{y}\rangle}p(x,\xi)q(y,\eta)\Big|_{x=y,\xi=\eta}\\
  =\sum_{k}\frac{h^{k}}{k!}\left(\frac{\langle{}D_{\xi},D_{y}\rangle}{i}\right)^{k}p(x,\xi)q(y,\eta)\Big|_{x=y,\xi=\eta.}\label{semiexp}\end{multline}
  \end{prop}
  
  \begin{prop}[Commutation identity]\label{prop:commute}
  Let $p(x,hD)$, $q(x,hD)$ be left-quantised semiclassical pseudodifferential operators. Then
  $$[p(x,hD),q(x,hD)]=hr(x,hD)$$
$$\norm{r(x,hD)}_{L^{2}\to{}L^{2}}\lesssim 1.$$
\end{prop}

\begin{prop}[Invertibility of elliptic operators]\label{prop:invert}
Let $p(x,hD)$ be a left-quantised, semiclassical pseudodifferential operator with symbol $p(x,\xi)$ such that $|p(x,\xi)|>c>0$. Then there exists an inverse operator $(p(x,hD))^{-1}$ with
$$\norm{(p(x,hD))^{-1}}_{L^{2}\to L^{2}}\lesssim 1.$$
\end{prop}

Suppose $\chi(x,\xi)$ is a smooth function and $p(x,hD)u=0$. Then from Proposition \ref{prop:commute} 
$$p(x,hD)\chi(x,hD)u=\chi(x,hD)p(x,hD)u+hr(x,hD)u=hr(x,hD)u$$
where $\norm{r(x,hD)}_{L^{2}\to{}L^{2}}\lesssim 1$. That is, the process of localisation reduces an exact solution to an approximate solution. Therefore we need to work with approximate solutions to $p(x,hD)u=0$ rather than exact ones.

\begin{defn}\label{defn:qm}
We say that $u$ is an order $h^{\beta}$ (sometimes written as $O_{L^{2}}(h^{\beta})$ or $O(h^{\beta})$) quasimode of $p(x,hD)$ if
$$\norm{p(x,hD)u}_{L^{2}}\lesssim{}h^{\beta}\norm{u}_{L^{2}}.$$
If $u$ is a joint order $h^{\beta}$ quasimode of $p_{1}(x,hD)\dots p_{r}(x,hD)$ then
$$\norm{p_{i}(x,hD)u}_{L^{2}}\lesssim{}h^{\beta}\norm{u}_{L^{2}}\quad i=1,\dots,r.$$
\end{defn}

Definition \ref{defn:qm} is enough to produce the $L^{p}$ estimates for quasimodes considered in \cite{koch}, \cite{tacy09} and \cite{HTacy}. However for this work we will need a slightly stronger kind of quasimode.  This issue arises as we could produce a quasimode $v$ from any exact solution $u$ by considering
$$v=u+hf$$
for some function $\norm{f}_{L^{2}}=1$. By choosing $f(x)=h^{-n/2}\chi(h^{-1}x)$ where $\chi$ is a compactly supported function we immediately see that we couldn't expect an $L^{\infty}$ estimate better than
$$\norm{v}_{L^{\infty}}\lesssim{}h^{-\frac{n-2}{2}}\norm{v}_{L^{2}}.$$
However this example is rather artificial. To deal with this we define the notion of a strong quasimode that has the property that repeated application of $p(x,hD)$ continues to improve the quasimode error.

\begin{defn}\label{defn:strongqm}
We say that $u$ is a strong order $h^{\beta}$ ($O^{str}_{L^{2}}(h^{\beta})$ or $O^{str}(h^{\beta})$) quasimode of $p(x,hD)$ if
$$\norm{p^{k}(x,hD)u}_{L^{2}}\lesssim{}h^{\beta k}\norm{u}_{L^{2}}\quad k=1,2,\dots$$
If $u$ is a  strong joint order $h^{\beta}$ quasimode of $p_{1}(x,hD)\dots p_{r}(x,hD)$ then
$$\norm{p^{k_{1}}_{1}(x,hD)\circ\cdots\circ p^{k_{r}}_{r}(x,hD)u}_{L^{2}}\lesssim{}h^{\beta (k_{1}+\cdots+k_{r})}\norm{u}_{L^{2}}\quad i=1,\dots, r, \; k_{i}=1,2,\dots$$
\end{defn}

Clearly an exact solution
$$p(x,hD)u=0$$
is a strong quasimode. Spectral clusters (a major example of quasimodes) are also strong quasimodes. Let
$$u=\sum_{\lambda_{j}\in[\lambda,\lambda-W(\lambda)]}\phi_{j}$$
where the $\phi_{j}$ are Laplacian eigenfunctions with eigenvalues $\lambda_{j}$ and $W(\lambda)\in[0,\lambda]$. Then
$$(-\Delta-\lambda^{2})u=\sum_{\lambda_{j}\in[\lambda,\lambda-W(\lambda)]}(\lambda_{j}-\lambda)(\lambda_{j}+\lambda)\phi_{j}.$$
So
$$\norm{-(\Delta-\lambda^{2})u}_{L^{2}}\lesssim W(\lambda)\lambda\norm{u}_{L^{2}}$$
and when rescaled to express this in terms of the semiclassical parameter $h=\lambda^{-1}$,
$$\norm{-(h^2\Delta-1)u}_{L^{2}}\lesssim W(h^{-1})h\norm{u}_{L^{2}}.$$
That is $u$ is an order $W(h^{-1})h$ quasimode. If we apply $(-\Delta-\lambda^{2})^{k}$ to $u$ we have
$$(-\Delta-\lambda^{2})^{k}u=\sum_{\lambda_{j}\in[\lambda,\lambda-W(\lambda)]}(\lambda_{j}-\lambda)^{k}(\lambda_{j}+\lambda)^{k}\phi_{j}$$
and rescaling $h=\lambda^{-1}$,
$$\norm{(-h^{2}\Delta-1)^{k}u}_{L^{2}}\lesssim W^{k}(h^{-1})h^{k}\norm{u}_{L^{2}}.$$
That is $u$ is a strong order $W(h^{-1})h$ quasimode.

We have seen that the commutation identity implies that the property of being an order $h$ quasimode is preserved under localisation.  That is if $u$ is an order $h$ quasimode of $p(x,hD)$, $\chi(x,hD)u$ is also an $O(h)$ quasimode of $p(x,hD)$. This property also holds for strong quasimodes.
\begin{prop}\label{prop:stronqm}
Suppose $u$ is a  strong joint order $h$ quasimode of $p_{1}(x,hD),\dots,p_{r}(x,hD)$  and $\chi(x,\xi)$ is a smooth compactly supported function on $T^{\star}M$. Then $\chi(x,hD)u$ is also a  strong joint order $h$ quasimode of $p_{1}(x,hD),\dots,p_{r}(x,hD)$.
\end{prop}
\begin{proof}
This is simply a repeated application of the commutation identity, 
\begin{align*}
p^{k_{1}}_{1}(x,hD) \circ\cdots \circ &p^{k_{r}}_{r}(x,hD)\chi(x,hD)u\\
&=p^{k_{1}}_{1}(x,hD)\circ\cdots\circ p_{r}^{k_{r}-1}(x,hD)\chi(x,hD)p_{r}(x,hD)u\\
&\qquad+hp^{k_{1}}_{1}(x,hD)\circ\cdots\circ p^{k_{r}-1}_{r}(x,hD)r_{0}(x,hD)u\\
&=p^{k_{1}}_{1}(x,hD)\circ\cdots\circ p^{k_{r}-2}_{r}(x,hD)\chi(x,hD)p^{2}_{r}(x,hD)u\\
&\qquad+hp_{1}^{k_{1}}(x,hD)\circ\cdots\circ p_{r}^{k_{r}-2}(x,hD)r_{1}(x,hD)p_{r}(x,hD)u\\
&\qquad\qquad+h^{2}p_{1}^{k_{1}}(x,hD)\circ\cdots\circ p^{k_{r}-2}_{r}(x,hD)\tilde{r}_{0}(x,hD)u\\
&\vdots\\
&=\sum_{i_{1}=1}^{k_{1}}\cdots\sum_{i_{r}=1}^{k_{r}}h^{k_{1}+k_{2}-\sum_{j=1}^{r}i_{j}}b_{i_{1},\dots,i_{r}}(x,hD)p_{1}^{i_{1}}(x,hD)\circ \cdots \circ p_{r}^{i_{r}}(x,hD)u\end{align*}
where each $b_{i_{1},\dots,i_{r}}(x,hD)$ has bounded mapping norm $L^{2}\to L^{2}$. Therefore
$$\norm{p_{1}^{k_{1}}(x,hD)\circ\cdots\circ p_{r}^{k_{r}}(x,hD)\chi(x,hD)u}_{L^{2}}\lesssim h^{k_{1}+\cdots+k_{r}}\norm{u}_{L^{2}}.$$
\end{proof}

We can use this localisation combined with invertibility properties of $p(x,hD)$ where $p(x,\xi)$ is elliptic to focus our attention of components of $u$ localised near the set 
$$\bigcap_{i=1}^{r}\{(x,\xi)\mid p_{i}(x,\xi)=0\}.$$
 From Proposition \ref{prop:invert} we know that if $|p_{i}(x,\xi)|>c>0$, the operator $p_{i}(x,hD)$ is invertible and its inverse $(p_{i}(x,hD))^{-1}$ has bounded mapping norm $L^{2}\to{}L^{2}$. Now consider $\chi(x,hD)u$ where $\chi(x,\xi)$ is supported near a point $(x_{0},\xi_{0})$ such that $p_{i}(x_{0},\xi_{0})\neq{}0$. By choosing the support of $\chi$ small enough we may assume that $p_{i}(x,\xi)$ is bounded away from zero on the support of $\chi$ and therefore so is $p^{k}_{i}(x,\xi)$. Proposition \ref{prop:com} tells us that $p^{k}_{i}(x,\xi)$ is the principal symbol of $p^{k}_{i}(x,hD)$ so by Proposition \ref{prop:invert} we can produce an inverse $(p^{k}_{i}(x,hD))^{-1}$. Therefore if 
$$p^{k}_{i}(x,hD)\chi(x,hD)u=h^{k}f,\quad\norm{f}_{L^{2}}\lesssim \norm{u}_{L^{2}}$$
we can invert $p_{i}(x,hD)$ to obtain
$$\chi(x,hD)u=h^{k}(p^{k}(x,hD))^{-1}f$$
and
 $$\norm{\chi(x,hD)u}_{L^{2}}\lesssim h^{k}\norm{u}_{L^{2}}.$$ 
 Now by applying semiclassical Sobolev estimates \cite{Zworski12} we obtain
$$\norm{\chi(x,hD)u}_{L^{p}}\lesssim{}h^{-\frac{n}{2}+\frac{n}{p}+k}\norm{u}_{L^{2}}.$$
By choosing $k$ large enough (dependent on $r$) we obtain better estimates than those of Theorem \ref{thm:main}. So we need only consider $\chi(x,hD)u$ where $\chi(x,\xi)$ is supported in a neighbourhood of some point $(x_{0},\xi_{0})$ where all the $p_{i}(x_{0},\xi_{0})=0$.

\section{Admissibility  Conditions}\label{sec:conditions}

In Theorems \ref{thm:main} and \ref{thm:n=r} we stated a set of admissibility conditions on the symbols of the operators $p_{i}(x,hD)$. This section is devoted to a discussion of the significance of these conditions. The first condition places a non-degeneracy assumption on the $p_{i}(x,\xi)$, namely that each $\{\xi\mid p(x_{0},\xi)=0\}$ is a smooth hypersurface. The second condition gives us information about how these hypersurfaces intersect.  To understand the importance of the intersection condition consider the following motivating example in $\R^{2}$ with $p_{1}(x,\xi)=|\xi|^{2}-1$.  Since this is a constant coefficient equation it is instructive to work on the Fourier side. In keeping with the semiclassical theme we use the semiclassical Fourier transform
$$\mathcal{F}_{h}[f](\xi)=\frac{1}{(2\pi h)^{n/2}}\int{}e^{-\frac{i}{h}\langle x,\xi\rangle }f(x)dx$$
where the prefactor is chosen so that $\norm{\mathcal{F}_{h}f}_{L^{2}}=\norm{f}_{L^{2}}$. Therefore to produce a strong quasimode for $p_{1}(x,hD)$ we need to solve the multiplier problem
$$(|\xi|^{2}-1)\mathcal{F}_{h}[u]=O_{L^{2}}(h\norm{\mathcal{F}_{h}[u]}_{L^{2}}).$$
Clearly any solution needs to be localised in an $h$ thickened annulus around $|\xi|=1$ (see Figure \ref{fig:FTann}).
\begin{figure}[h!]
\includegraphics[scale=0.4]{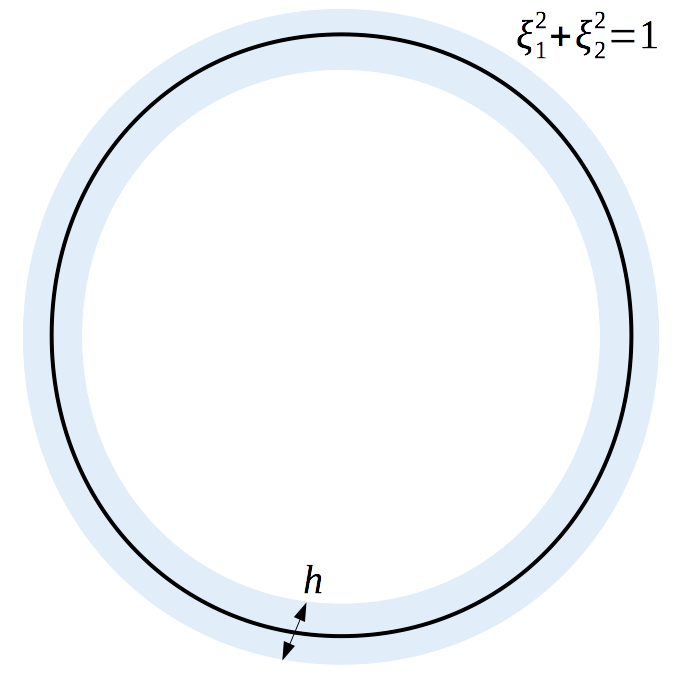}
\caption{The Fourier transform of $u$ must be located in an $O(h)$ annulus around $|\xi|=1$}
\label{fig:FTann}
\end{figure}
Now we ask, what further restrictions on the support of $\mathcal{F}_{h}[u]$ would force the $L^{\infty}$ norm of $u$ to be small? To have a large $L^{\infty}$ norm we must concentrate (about a single point) as much of the $L^{2}$ mass as possible. The uncertainty principle tells us that such intense spatial concentration is to be achieved by spreading the $L^{2}$ mass of the Fourier transform as much as possible. Conversely concentrating the mass of the Fourier transform will force $u$ to spread out, reducing the $L^{\infty}$ norm. 

Therefore to gain an improvement we need to set $p_{2}(x,\xi)$ in such a way that we force $\mathcal{F}_{h}[u]$ to be supported in a smaller region. An immediate choice is $p_{2}(x,\xi)=\xi_{2}$. Strong quasimodes to this equation require
$$\xi_{2}\mathcal{F}_{h}[u]=O_{L^{2}}(h\norm{\mathcal{F}_{h}[u]}_{L^{2}})$$
and therefore must have their Fourier transform located within a distance $h$ from the hypersurface $\xi_{2}=0$. To obey both requirements $\mathcal{F}_{h}[u]$ must be located in an $O(h)$ size ball about either $(-1,0)$ or $(1,0)$ (as shown in Figure \ref{fig:joint}). 
\begin{figure}[h!]
\includegraphics[scale=0.4]{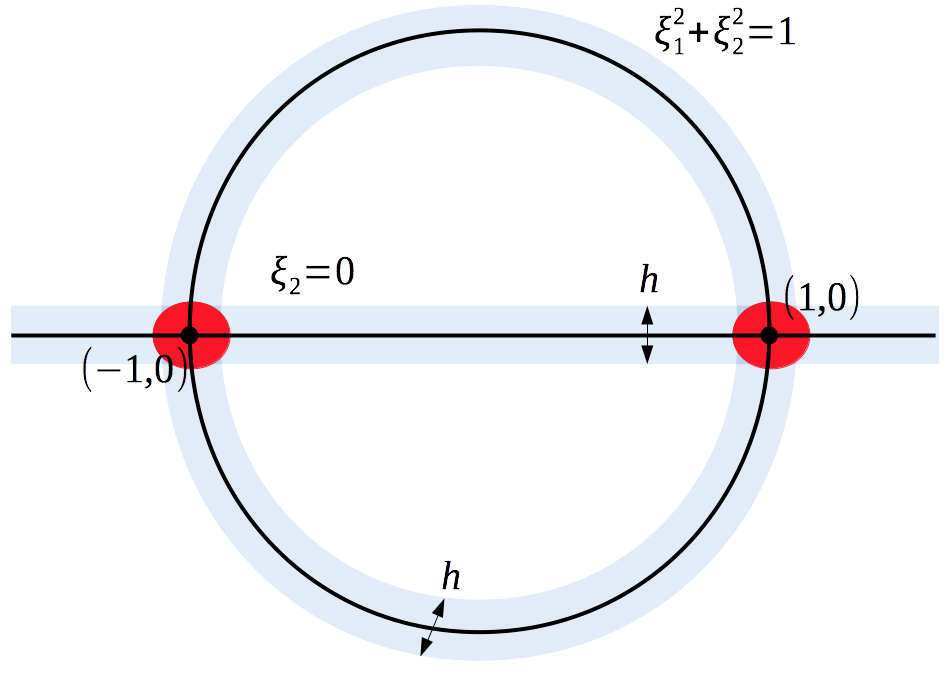}
\caption{With the additional requirement that $u$ be a strong quasimode of $p_{2}(x,hD)=hD_{x_{2}}$ we find that the Fourier transform of $u$ must be located in an $O(h)$ ball about $(-1,0)$ or $(1,0)$.}
\label{fig:joint}
\end{figure}
The uncertainty principle then tells us that $u$ will be spread across a region of size $1$, therefore its $L^{\infty}$ and $L^{2}$ norms will be comparable. 

To obtain the admissibility conditions of Theorems \ref{thm:main} and \ref{thm:n=r} consider the information provided by each quasimode equation at  a point on the intersection of  $\{\xi\mid p_{1}(\xi)=0\}$ and $\{\xi\mid p_{2}(\xi)=0\}$, for example $(1,0)$. The quasimode equation derived from $p_{1}(hD_{x})$ tells us that we may ``smear'' the mass of $\mathcal{F}_{h}[u]$ for an order $h$ region in the normal direction to $|\xi|^{2}=1$. Similarly we may ``smear'' the mass of $\mathcal{F}_{h}[u]$ for an order $h$ region in the normal direction to $\xi_{2}=0$. Since the normal vectors form a spanning set taking both requirements together restrict us to an order $h$ ball about $(1,0)$. The conditions of Theorems \ref{thm:main} and \ref{thm:n=r} generalise this by requiring the normal vectors be linearly independent thus when we add an additional $p_{i}(x,\xi)$ we add another direction in which the Fourier transform of $u$ is controlled. 

One immediate question posed by this heuristic is, what happens if the characteristic sets $\{\xi\mid p_{1}(\xi)=0\}$ and $\{\xi\mid p_{2}(\xi)=0\}$ have higher order contact? On the sphere $S^{2}$ the highest weight harmonics satisfy
$$p_{1}(x,hD)u=(h^{2}\Delta-1)u=0\quad{}p_{2}(x,hD)u=(h^{2}D^{2}_{\varphi}-h^{2}m^{2})u=0$$
where $h^{-2}=m(m+1)$. Such functions have $L^{\infty}$ norm that grows as $h^{-1/4}$. The characteristic sets in this case have order one contact. The general quantum completely integrable system lower bounds of \cite{TZ1}, \cite{TZ2} and \cite{TZ} follow a similar pattern are also associated with order one contact. If we apply the heuristic that the order of contact between characteristic sets near a point $(x_{0},\xi_{0})$ tells us the degree to which we can ``smear'' the Fourier transform this suggests that contact of order $k$ would allow a smearing of order $h^{\frac{1}{k+1}}$. This would imply an improvement (over the general $L^{\infty}$ norm bounds) of $h^{\frac{1}{2(k+1)}}$ for each direction (given by an additional operator). Certainly this is borne out in the known order one contact cases. However we leave this as a conjecture for future study.

The distinctive piecewise linear form of Sogge's $\delta(n,p)$ arises due to the curvature of $|\xi|_{g(x)}=1$. Indeed these questions regarding the growth of eigenfunctions are closely related to the classical harmonic analysis theory of the restriction operator and its adjoint and rely on the same type of curvature assumptions. It is therefore this curvature that we seek to replicate with the curvature condtion (3). In \cite{koch} their curvature condition on $p_{1}(x,hD)$ was that
\begin{itemize}
\item For each $x_{0}$ the set $\{\xi\mid{}p_{1}(x_{0},\xi)=0\}$ has nonzero Gauss curvature. 
\end{itemize}
To obtain the hypersurface estimates of \cite{tacy09} it was necessary to strengthen the second condition to
\begin{itemize}
\item For each $x_{0}$ the set $\{\xi\mid{}p_{1}(x_{0},\xi)=0\}$ has positive definite second fundamental form. 
\end{itemize}
The strengthened condition was necessary to deal with taking cross sections of $\{\xi\mid{}p_{1}(x_{0},\xi)=0\}$ and requiring those cross sections to display curvature. Similarly conditions (3) required that $\cap_{j}\{\xi\mid p(x_{0},\xi)=0\}$ display curvature. If $\{\xi\mid p_{1}(x_{0},\xi)=0\}$ is positive definite then condition (3) is met independent of the other $p_{i}(x,\xi)$. So it is sufficient to have one of the operators Laplace-like.

The admissibility conditions for the symbols $p_{j}(x,\xi)$ of Theorems \ref{thm:main} and \ref{thm:n=r} are stated in terms of symbols $p_{j}(x,\xi)$ that are independent of $h$.  There are, however, cases where we may wish to consider $h$ dependent symbols. For example the ladder operators of \cite{TZ}
$$p_{j}(x,\xi;h)=p_{j}(x,\xi)-E_{j}(h).$$
The results of this paper still hold for such symbols so long as both the symbols have uniform regularity in $h$ and the geometric admissibility conditions hold uniformly. 

\section{$L^{p}$ estimates on joint quasimodes}\label{sec:jointest}

In this section we focus on proving Theorems \ref{thm:main} and \ref{thm:n=r}. We have seen that for strong quasimodes we need only consider contributions that are semiclassically localised near points $(x_{0},\xi_{0})$ that lie in the intersections of the characteristic sets of the $p_{j}(x,\xi)$. That is we want to establish Theorem \ref{thm:localmain}.

\begin{thm}\label{thm:localmain}
Let $r<n$. Suppose $u$ is a strong joint order $h$ quasimode of $p_{1}(x,hD),\dots, p_{r}(x,hD)$ satisfying the admissibility conditions of Theorem \ref{thm:main}. Let $\chi(x,\xi)$ be a smooth compactly supported function localised near a point $(x_{0},\xi_{0})$ at which $p_{j}(x_{0},\xi_{0})=0$ for all $j$. Then
\begin{equation}\norm{\chi(x,hD)u}_{L^{p}}\lesssim{}h^{-\delta(n,p,r)}\norm{u}_{L^{2}}\label{Lpest}\end{equation}
\begin{equation}
\delta(n,p,r)=\begin{cases}
\frac{n-r}{2}-\frac{n-r+1}{p}&\frac{2(n-r+2)}{n-r}\leq{}p\leq\infty\\
\frac{n-r}{4}-\frac{n-r}{2p}&2\leq{}p\leq\frac{2(n-r+2)}{n-r}.\end{cases}\label{deltanpr}\end{equation}
\end{thm}

We prove Theorem \ref{thm:localmain} in the following three steps.
\begin{itemize}
\item[\textbf{Step 1}] Since each characteristic set $\{\xi\mid p(x,\xi)=0\}$ is a smooth hypersurface we can write it as a locally graph. In particular, after a careful choice of coordinate system, we can write
$\{\xi\mid p_{r}(x,\xi)=0\}$ as the graph
$$\xi_{r}=a_{r}(x,\xi_{1},\dots,\xi_{r-1},\xi_{r+1},\dots,\xi_{n})$$
for some $a_{r}$. We can then ``factor'' $\xi_{r}$ out of the other $p_{j}(x,\xi)$ by substituting $a_{r}$ for $\xi_{r}$.  In Proposition \ref{prop:factor} we use this idea to define an inductive process to factor out $\xi_{2},\dots,\xi_{r}$. 
\item[\textbf{Step 2}] We are left with a semiclassical equation that does not involve derivatives in $x_{2},\dots,x_{r}$. In Proposition \ref{prop:KTZ} we treat this as a $n-r+1$ dimensional semiclassical quasimode and apply the results of Koch-Tataru-Zworski \cite{koch}. 
\item[\textbf{Step 3}] Finally we need to estimate the $L^{2}\to L^{p}$ growth for $x_{2},\dots, x_{r}$. In Proposition \ref{prop:timeslice} we do this by using each $p_{j}(x,\xi)$, $j=2,\dots,r$ to produce an evolution equation for which $x_{j}$ behaves as the time variable. 

\end{itemize}

To facilitate this process we need to introduce some notation to express the removal of various $\xi_{i}$
\begin{defn}
For $\xi\in\R^{n}$ we write
$$\tilde{\xi}^{(i)}=(\xi_{1},\dots,\xi_{i-1},\xi_{i+1},\dots,\xi_{n})\in \R^{n-1}$$
and 
\begin{align*}\tilde{\xi}^{(i,j)}&=(\xi_{1},\dots,\xi_{i-1},\xi_{j+1},\dots,\xi_{n})\in \R^{n-(j-i+1)}\quad \text{for $i<j$},\\
\tilde{\xi}^{(i,i)}&=\tilde{\xi}^{(i)}.\end{align*}
\end{defn}
A key part of the proof is Proposition \ref{prop:factor} which tells us how to factor out the variables $\xi_{2},\dots,\xi_{r}$. Before we prove the general case we will look at an explicit example with $n=3,r=2$ to fix our ideas.

\begin{example}\label{ex:factor}
Let 
$$p_{1}(x,\xi)=p_{1}(\xi)=|\xi|^{2}-1\quad\text{and}\quad p_{2}(x,\xi)=\xi_{1}+\xi_{2}-\xi_{3}+x^{2}_{2}.$$
We localise to the region near the point $x_{0}=(0,0,0),\xi_{0}=(1,0,0)$. Note that at this point $\nu_{1}=(1,0,0)$ and $\nu_{2}=(1,1,-1)$. So as long as we are suitably localised near $(x_{0},\xi_{0})$ the linear independence of $\nu_{1}$ and $\nu_{2}$ is guaranteed. Further $\{\xi\mid p_{1}(\xi)=0\}$ has positive definite second fundamental form so the curvature condition is satified. We need first to pick a good coordinate system in which to work. We have
$$\partial_{\xi_{1}}p_{1}(x_0,\xi_0)=2\quad{}\nabla_{\tilde{\xi}^{1}}p_{1}(x_{0},\xi_{0})=0$$
so we will make no changes that involve $\xi_{1}$. Now
$$\partial_{\xi_{2}}p_{2}(x_{0},\xi_{0})=1\quad{}\partial_{\xi_{3}}p_{2}(x_{0},\xi_{0})=-1$$
so we make a change of coordinates such that in the new system
$$\partial_{\xi_{2}}p_{2}(x_{0},\xi_{0})\neq 0\quad{}\partial_{\xi_{3}}p_{2}(x_{0},\xi_{0})=0.$$
For example 
$$\frac{1}{2}\left(\begin{array}{cc}
1&-1\\
1&1\end{array}\right)\left[\begin{array}{c}
\xi_{2}\\
\xi_{3}\end{array}\right]$$
is suitable and, under this change, $p_{2}(x,\xi)$ becomes
$$p_{2}(x,\xi)=\xi_{1}+2\xi_{2}+x^{2}_{2}.$$
Now we are in a position to factor out $\xi_{2}$. We write
$$p_{2}(x,\xi)=2\left(\xi_{2}+\frac{\xi_{1}+x^{2}_{2}}{2}\right)$$
so that the zero set of $p_{2}(x,\xi)$ is given by $\xi_{2}=-\frac{\xi_{1}+x^{2}_{2}}{2}$. We then produce a new symbol which we denote as $\tilde{p}_{1}^{(2)}(x,\tilde{\xi}^{(2)})$,
$$\tilde{p}_{1}^{(2)}(x,\tilde{\xi}^{(2)})=p_{1}\left(x,\xi_{1},-\frac{\xi_{1}+x^{2}_{2}}{2},\xi_{3}\right)=\frac{5}{4}\xi^{2}_{1}+\xi^{2}_{3}+\frac{x_{2}^{2}\xi_{1}}{2}+\frac{x_{2}^{4}}{4}.$$
Note that for any fixed $x$ near $x=0$, $\{\xi\mid \tilde{p}^{2}_{1}(x,\tilde{\xi}^{(2)})=0\}$ still has positive definite second fundamental form when considered as a hypersurface in $\R^{2}$. In Proposition \ref{prop:factor} we see how to move through this process in the general case.
\end{example}

\begin{prop}\label{prop:factor}
Suppose $p_{1}(x,hD),\dots,p_{r}(x,hD)$ satisfy the admissibility conditions of Theorem \ref{thm:main}. Then  for each $k=0,\dots,r-1$ there exists a set of symbols 
$$\tilde{p}^{(r-k,r)}_{i}(x,\tilde{\xi}^{(r-k,r)})\quad{} i=1,\dots,r-k-1$$
where
\begin{itemize}
\item For any $x$, 
$$\bigcap_{i=1}^{r-k-1}\{\xi\mid \tilde{p}_{i}^{(r-k,r)}(x,\tilde{\xi}^{(r-k,r)})=0\}=\bigcap_{i=1}^{r}\{\xi\mid p_{i}(x,\xi)=0\}$$
\item $\partial_{\xi_{i}}\tilde{p}^{(r-k,r)}_{i}(x_{0},\tilde{\xi}^{(r-k,r)}_{0})\neq{}0,\quad{}\partial_{\xi_{j}}\tilde{p}^{(r-k,r)}_{i}(x_{0},\tilde{\xi}^{(r-k,r)}_{0})=0 \quad j>i.$
\item For $k=0,\dots,r-2$
$$\partial^{2}_{\xi_{i}\xi_{j}}\tilde{p}^{(r-k,r)}_{1}(x_{0},\tilde{\xi}^{(r-k,r)}_{0})\quad i,j=r+1,\dots, n$$
 is a non-degenerate matrix.
\end{itemize}
\end{prop}

\begin{proof}
Assume that the $p_{i}(x,\xi)$ are numbered so that the shape operator associated with  $\{\xi\mid p_{1}(x_{0},\xi)=0\}$ satisfies  the curvature condition. The linear independence of the $\nu_{j}(x_{0},\xi_{0})$  means that there is a linear map from the $\nu_{j}$ to the first $r$ basis vectors of $\R^{n}$. We us this to pick an appropriate coordinate system in which to work. Since the set $\{\xi\mid p_{1}(x_{0},\xi)=0\}$ is a hypersurface we know that
$$\nabla_{\xi}p_{1}(x_{0},\xi_{0})\neq{}0.$$
Therefore there is some $\xi_{i}$ so $\partial_{\xi_{i}}p_{1}(x_{0},\xi_{0})\neq{}0$. A suitable change of coordinate system sets this $\xi_{i}$ to $\xi_{1}$ and we have
$$\partial_{\xi_{1}}p_{1}(x_{0},\xi_{0})\neq{}0\quad\text{and}\quad{}\nabla_{\tilde{\xi}^{(1)}}p_{1}(x_{0},\xi_{0})=0.$$

We will now set $\xi_{2},\dots,\xi_{r}$ as determined by $p_{2}(x,\xi),\dots,p_{r}(x,\xi)$. Since $\{\xi\mid p_{2}(x_{0},\xi)=0\}$ is a hypersurface there is some $\xi_{i}$ so that
$$\partial_{\xi_{i}}p_{2}(x_{0},\xi_{0})\neq{}0$$
and the linear independence of the normals tells us that, $\langle \xi_{i},\tilde{\xi}^{(1)}\rangle\neq{}0$. That is we can set $\xi_{2}$ so that
$$\partial_{\xi_{2}}p_{2}(x_{0},\xi_{0})\neq{}0\quad \nabla_{\tilde{\xi}^{(1,2)}}p(x_{0},\xi_{0})=0.$$
Continuing in this fashion we have
\begin{equation}\partial_{\xi_{i}}p_{i}(x_{0},\xi_{0})\neq{}0\quad{}\nabla_{\tilde{\xi}^{(1,i)}}p(x_{0},\xi_{0})=0\label{nondeghy}\end{equation}
for  $i=1,\dots,r$. 
 
 At $(x_{0},\xi_{0})$ in this coordinate system
 $$\gamma S_{1}\gamma^{\star}=\partial^{2}_{\xi_{i}\xi_{j}}p_{1}(x_{0},\xi_{0})\quad i,j=r+1,\dots,n.$$
 So the curvature condition (3) of Theorem \ref{thm:main} implies that 
 $$\partial^{2}_{\xi_{i}\xi_{j}}p_{1}(x_{0},\xi_{0})\quad i,j=r+1,\dots,n$$ 
 is non-degenerate. By choosing the support of the localiser $\chi(x,\xi)$ small we may assume that this non-degeneracy persists on  the support of $\chi$. 

We now define a process to inductively remove each $\xi_{i}$, $i=2,\dots,r$. From \eqref{nondeghy}, using the implicit function theorem, we can write the set $\{\xi\mid p_{r}(x,\xi)=0\}$ as a graph $\xi_{r}=a_{r}(x,\tilde{\xi}^{(r)})$ and factorise  $p_{r}(x,\xi)$ as
$$p_{r}(x,\xi)=e_{r}(x,\xi)(\xi_{r}-a_{r}(x,\tilde{\xi}^{(r)}))\quad |e_{r}(x,\xi)|\geq{}c>0.$$
 We now substitute the expression $\xi_{r}=a_{r}(x,\tilde{\xi}^{(r)})$ into each of the $p_{i}(x,\xi)$ for $i=1,\dots,r-1$ and therefore produce a set of symbols $\tilde{p}^{(r)}_{i}(x,\tilde{\xi}^{(r)})$ which are independent of $\xi_{r}$ but preserve the intersection of characteristic sets. That is
  $$\tilde{p}^{(r)}_{i}(x,\tilde{\xi}^{(r)})=p_{i}(x,\xi_{1},\dots,\xi_{r-1},a_{r}(x,\tilde{\xi}^{(r)}),\xi_{r+1},\dots,\xi_{n}).$$
  Now
$$\partial_{\xi_{j}}\tilde{p}^{(r)}_{i}(x,\tilde{\xi}^{(r)})=\partial_{\xi_{j}}p_{i}(x,\tilde{\xi}^{(r)})+\partial_{\xi_{r}}p_{i}(x,\tilde{\xi}^{(r)})\partial_{\xi_{j}}a_{r}(x,\tilde{\xi}^{(r)})$$
so
$$\partial_{\xi_{j}}\tilde{p}^{(r)}_{i}(x_{0},\tilde{\xi}_{0}^{(r)})=0\quad{}j>i$$
and
$$\partial_{\xi_{i}}\tilde{p}^{(r)}_{i}(x_{0},\tilde{\xi}_{0}^{(r)})\neq{}0.$$
We now need to check that 
$$\partial_{\xi_{i}\xi_{j}}^{2}\tilde{p}_{1}^{(r)}(x_{0},\tilde{\xi}^{(r)}_{0})\quad i=r+1,\dots,n$$
is non-degenerate
\begin{multline*}\partial^{2}_{\xi_{i}\xi_{j}}\tilde{p}^{(r)}_{1}(x,\tilde{\xi}^{(r)})=\partial^{2}_{\xi_{i}\xi_{j}}p_{1}(x,\tilde{\xi}^{(r)})+\partial^{2}_{\xi_{i}\xi_{r}}p_{1}(x,\tilde{\xi}^{(r)})\partial_{\xi_{j}}a_{r}(x,\tilde{\xi}^{(r)})\\
+\partial^{2}_{\xi_{j}\xi_{r}}p_{1}(x,\tilde{\xi}^{(r)})\partial_{\xi_{i}}a_{1}(x,\tilde{\xi}^{(r)})+\partial_{\xi_{r}}p_{1}(x,\tilde{\xi}^{(r)})\partial^{2}_{\xi_{i}\xi_{j}}a_{r}(x,\tilde{\xi}^{(r)}).\end{multline*}
At $(x_{0},\xi_{0})$
$$\partial_{\xi_{j}}\tilde{p}_{r}(x_{0},\xi_{0})=e_{r}(x_{0},\xi_{0})\partial_{\xi_{j}}a(x_{0},\xi_{0}),$$
so if $j>r$ we have
$$\partial_{\xi_{j}}a(x_{0},\xi_{0})=0.$$
Therefore for $i,j>r$
$$\partial^{2}_{\xi_{i}\xi_{j}}\tilde{p}^{(r)}_{1}(x_{0},\tilde{\xi}_{0}^{(r)})=\partial^{2}_{\xi_{i}\xi_{j}}p_{1}(x_{0},\xi_{0})$$
and is non-degenerate. Again by taking a region of small support around $(x_{0},\xi_{0})$ we may assume this holds on the support of $\chi$. 

We can now repeat the process to remove $\xi_{r-1}$. Note that we have  $\partial_{\xi_{r-1}}\tilde{p}^{(r)}_{r-1}(x_{0},\xi_{0})\neq{}0$ so we write $\{\xi\mid \tilde{p}^{(r)}_{r-1}(x,\tilde{\xi}^{(r)})=0\}$ as a graph $\xi_{r-1}=a_{r-1}(x,\tilde{\xi}^{(r-1,r)})$. We can then produce $\tilde{p}^{(r-1,r)}_{i}(x,\tilde{\xi}^{(r-1,r)})$ for $i=1,\dots,r-2$ in the same fashion as we produced the $\tilde{p}^{(r)}_{i}(x,\tilde{\xi}^{(r)})$. By  continuing inductively we produce a $\tilde{p}^{(r-k,r)}_{i}(x,\tilde{\xi}^{(r-k,r)})$ as required.

\end{proof}

At the final step of the inductive process of Proposition \ref{prop:factor} we produce a $\tilde{p}^{(2,r)}_{1}(x,\tilde{\xi}^{(2,r)})$ so that
$$\partial_{\xi_{1}}\tilde{p}^{(2,r)}_{1}(x_{0},\tilde{\xi}_{0}^{(2,r)})\neq{}0$$
and the matrix
$$\frac{\partial^{2}\tilde{p}^{(2,r)}_{1}}{\partial\xi_{i}\partial\xi_{j}}\quad i,j=r+1,\dots,n$$
is non-degenerate. So a final application of the implicit function theorem tells us that there is some $b(x,\tilde{\xi}^{(1,\dots,r)})$ so that
\begin{equation}\tilde{p}^{(2,r)}_{1}(x,\tilde{\xi}^{(2,r)})=e_{1}(x,\tilde{\xi}^{(2,r)})(\xi_1-b(x,\tilde{\xi}^{(1,r)})).\label{factored}\end{equation}
For our future computations we adopt the more convenient notation that $\tilde{\xi}^{(1,r)}=\eta\in\R^{n-r}$, $x_{1}=t$ and $x=(t,y,z)$ where $z$ is dual to $\eta$. Again by writing
$$\tilde{p}^{(2,r)}_{1}(t,y,z,\xi_{1},\eta)=e_{1}(t,y,z,\xi_{1},\eta)(\xi_{1}-b(t,y,z,\eta))$$
with $|e_{1}(t,y,z,\xi_{1},\eta)|>c>0$ we see that $\partial^{2}_{\eta_{i}\eta_{j}}b$ is a non-degenerate matrix so long as $\chi$ is supported in a sufficiently small region about $(x_{0},\xi_{0})$.

\begin{prop}\label{prop:KTZ}
Let $r<n$ and $E_{1}[u]$ be the quasimode error of $\chi(x,hD)u$ with respect to $(hD_{t}-b(t,y,z,hD_{z}))$. That is
\begin{equation}E_{1}[u]=(hD_{t}-b(t,y,z,hD_{z}))u\label{evoleq}\end{equation}
and assume
$$\frac{\partial^{2}b}{\partial\eta_{i}\partial\eta_{j}}\quad\text{ is a non-degenerate matrix}$$
and
$$\nabla_{\eta}b(t_{0},y_{0},z_{0},\eta_{0})=0.$$
Then
$$\norm{u}_{L^{p}_{y}L^{p}_{t}L^{p}_{z}}\lesssim h^{-\delta(n,p,r)}\left(\norm{u}_{L^{p}_{y}L^{2}_{t}L^{2}_{z}}+h^{-1}\norm{E_{1}[u]}_{L^{p}_{y}L^{2}_{t}L^{2}_{z}}\right).$$
\end{prop}

\begin{proof}
Consider \eqref{evoleq} as an inhomogeneous evolution equation where $y$ acts as a parameter. That is
$$(hD_{t}-b_{y}(t,z,hD_{z}))u=E_{1}[u]$$
and $\partial^{2}_{\eta_{i}\eta_{j}}b$ is non-degenerate. This is exactly the kind of quasimode treated in Koch-Tataru-Zworski \cite{koch} (Section 5 in particular Theorem 5) with dimension equal to $n-r+1$. Applying their results for fixed $y$ we obtain
\begin{equation}\norm{u(\cdot,y,\cdot)}_{L^{p}}\lesssim{}h^{-\delta(n,p,r)}\left(\norm{u(\cdot,y,\cdot)}_{L^{2}}+h^{-1}\norm{E_{1}[u](\cdot,y,\cdot)}_{L^{2}_{t}L^{2}_{z}}\right).\label{KTZ}\end{equation}
The constant implicit in \eqref{KTZ} may depend on the size of $b_{y}(t,z,\eta)$ and its derivatives. However since we are on a compact region of phase space we can take a supremum over $y\in \text{supp}(\chi)$ and obtain \eqref{KTZ} with a uniform constant. Therefore taking $L^{p}$ norms in $y$ we obtain
$$\norm{u}_{L^{p}_{y}L^{p}_{t}L^{p}_{z}}\lesssim h^{-\delta(n,p,r)}\left(\norm{u}_{L^{p}_{y}L^{2}_{t}L^{2}_{z}}+h^{-1}\norm{E_{1}[u]}_{L^{p}_{y}L^{2}_{t}L^{2}_{z}}\right).$$
\end{proof}
 
 So to obtain Theorem \ref{thm:localmain} we need only prove that
 $$\left(\norm{u}_{L^{p}_{y}L^{2}_{t}L^{2}_{z}}+h^{-1}\norm{E_{1}[u]}_{L^{p}_{y}L^{2}_{t}L^{2}_{z}}\right)\lesssim \norm{u}_{L^{2}_{y}L^{2}_{t}L^{2}_{z}}.$$
 We divide into the cases $r=2$ and $r\geq{}3$. The ideas in each case are identical however the notation is slightly different. 
 
  \begin{prop}\label{prop:timeslice2} Suppose $r=2$. Under the assumptions of Theorem \ref{thm:localmain} and in the coordinate system developed in Proposition \ref{prop:factor}
  $$\left(\norm{u}_{L^{p}_{y}L^{2}_{t}L^{2}_{z}}+h^{-1}\norm{E_{1}[u]}_{L^{p}_{y}L^{2}_{t}L^{2}_{z}}\right)\lesssim \norm{u}_{L^{2}_{y}L^{2}_{t}L^{2}_{z}}.$$
 \end{prop}
 
 \begin{proof}
Consider $p_{2}(x,\xi)$. We can factorise the symbol
$$p_{2}(x,\xi)=e_{2}(x,\xi)(\xi_{2}-a_{2}(x,\tilde{\xi}^{(2)}))$$
with $|e_{2}(x,\xi)|>c>0$. So if $u$ is an order $h$ quasimode of $p_{2}(x,hD)$, then $u$ is also an order $h$ quasimode of
$$(hD_{x_{2}}-a_{2}(x,hD_{\tilde{x}^{(2)}})).$$
In $(t,y,z)$ coordinates 
$$(hD_{y_{1}}-a_{2}(t,y,z,hD_{t},hD_{z}))u=E_{2}[u],\quad\norm{E_{2}[u]}_{L^2}\lesssim\norm{p_{2}(x,hD)u}_{L^{2}}.$$
So as in Koch-Tataru-Zworski \cite{koch} we use Duhamel's principle. We can write
$$u=U_{2}(y_{1},0)u(t,0,z)+\frac{i}{h}\int_{0}^{y_{1}}U_{2}(y_{1},s)E_{2}[u]ds$$
where
$$\begin{cases}
(hD_{y_{1}}-a_{2}(t,y_{1},z,hD_{t},hD_{z}))U_{2}(y_{1},s)=0\\
U_{2}(s,s)=\Id.\end{cases}$$
So
$$\norm{U(\cdot,0)u|_{y_{1}=0}}_{L^{2}_{y_{1}}L^{2}_{t}L^{2}_{z}}\lesssim\left(\norm{u}_{L^{2}_{y_{1}}L^{2}_{t}L_{z}^{2}}+h^{-1}\norm{E_{2}[u]}_{L^{2}_{y_{1}}L^{2}_{t}L^{2}_{z}}\right).$$
Since $U_{2}$ is unitary 
$$\norm{u(\cdot,0,\cdot)}_{L^{2}_{t}L^{2}_{z}}\lesssim\norm{U(\cdot,0)u|_{y_{1}=0}}_{L^{2}_{y_{1}}L^{2}_{t}L^{2}_{z}}.$$
Then we may conclude that
\begin{equation}\norm{u(\cdot,0,\cdot)}_{L^{2}_{t}L^{2}_{z}}\lesssim\left(\norm{u}_{L^{2}_{y_{1}}L^{2}_{t}L_{z}^{2}}+h^{-1}\norm{E_{2}[u]}_{L^{2}_{y_{1}}L^{2}_{t}L^{2}_{z}}\right)\label{restrict}\end{equation}
and indeed by shifting the zero in $y_{1}$ \eqref{restrict} is true for any $\norm{u(\cdot,y_{1},\cdot)}_{L^{2}_{t}L^{2}_{z}}$. Therefore
$$\norm{u}_{L^{\infty}_{y_{1}}L^{2}_{t}L^{2}_{z}}\lesssim\left(\norm{u}_{L^{2}_{y_{1}}L^{2}_{t}L_{z}^{2}}+h^{-1}\norm{E_{2}[u]}_{L^{2}_{y_{1}}L^{2}_{t}L^{2}_{z}}\right).$$
and since $y_{1}$ lies in a compact set,
$$\norm{u}_{L^{p}_{y_{1}}L^{2}_{t}L^{2}_{z}}\lesssim\left(\norm{u}_{L^{2}_{y_{1}}L^{2}_{t}L_{z}^{2}}+h^{-1}\norm{E_{2}[u]}_{L^{2}_{y_{1}}L^{2}_{t}L^{2}_{z}}\right).$$
By treating $E_{1}[u]$ itself as a quasimode we also have
$$\norm{E_{1}[u]}_{L^{p}_{y_{1}}L^{2}_{t}L^{2}_{z}}\lesssim\left(\norm{E_{1}[u]}_{L^{2}_{y_{1}}L^{2}_{t}L_{z}^{2}}+h^{-1}\norm{E_{2}\left[E_{1}[u]\right]}_{L^{2}_{y_{1}}L^{2}_{t}L^{2}_{z}}\right).$$
 So we need only show that 
 $$\norm{E_{2}[u]}_{L^{2}}\leq{}h\norm{u}_{L^{2}}$$
$$\norm{E_{1}[u]}_{L^{2}}\leq{}h\norm{u}_{L^{2}}$$
and
$$\norm{E_{2}\left[E_{1}[u]\right]}_{L^{2}}\leq{}h^{2}\norm{u}_{L^{2}}.$$
The first inequality follows from $u$ being an order $h$ quasimode of $p_{2}(x,hD)$ and the fact that we can write
$$p_{2}(x,\xi)=e_{2}(x,\xi)(\xi_{2}-a_{2}(x,\tilde{\xi}^{(2)})$$
where $|e_{2}(x,\xi)|>c>0$.  The second two inequalities follow from the claim that if $u$ is a strong joint order $h$ quasimode of $q(x,hD),p(x,hD)$ and $(hD_{x_{2}}-a(x,hD_{\tilde{x}^{(2)}}))$ that $u$ is a strong joint order $h$ quasimode of $q(x,hD)$ and $\tilde{p}^{(2)}(x,hD_{\tilde{x}^{(2)}})$ where
$$\tilde{p}^{(2)}(x,\tilde{\xi}^{(2)})=p(x,\xi_{1},a(x,\tilde{\xi}^{(2)}),\xi_{i+1},\dots,\xi_{n}).$$
Consider the difference $p(x,\xi)-\tilde{p}^{(2)}_{1}(x,\tilde{\xi}^{(2)})$ and expand in $\xi_{2}$ about $a(x,\tilde{\xi}^{(2)})$. We obtain
$$p(x,\xi)-\tilde{p}^{(2)}(x,\tilde{\xi}^{(2)})=(\xi_{2}-a(x,\tilde{\xi}^{(2)}))r(x,\xi).$$
So
$$q^{k_{1}}(x,hD)(\tilde{p}^{(2)}(x,hD_{\tilde{x}^{(2)}}))^{k_{2}}=q^{k_{1}}(x,hD)(p(x,hD)-(hD_{x_{2}}-a(x,hD_{\tilde{x}^{(2)}}))r(x,hD))^{k_{2}}$$
and so expanding via the binomial formula
$$\norm{q^{k_{1}}(x,hD)(\tilde{p}^{(2)}(x,hD))^{k_{2}}u}\lesssim h^{k_{1}+k_{2}}\norm{u}_{L^{2}}.$$
Therefore since $u$ is a strong joint order $h$ quasimode of $p_{1}(x,hD)$ and $p_{2}(x,hD)$ we know that
$$\norm{\tilde{p}^{(2)}_{1}(x,hD_{\tilde{x}^{(2)}})\circ p_{2}(x,hD)u}_{L^{2}}\lesssim{}h^{2}\norm{u}_{L^{2}}$$
and
$$\norm{\tilde{p}^{(2)}_{1}(x,hD_{\tilde{x}^{(2)}})u}_{L^{2}}\lesssim h\norm{u}_{L^{2}}.$$
However
$$E_{i}[v]=(hD_{x_{i}}-a_{i}(x,hD_{\tilde{x}^{(i)}}))v\quad{}i=1,2$$
where
$$p_{i}(x,\xi)=e_{i}(x,\xi)(\xi_{i}-a_{i}(x,\tilde{\xi}^{(i)}))$$
and $|e_{i}(x,\xi)|>c>0$ so by the invertibility of the $e_{i}(x,hD)$ we obtain
$$\norm{E_{1}[u]}_{L^{2}}\leq{}h\norm{u}_{L^{2}}$$
and
$$\norm{E_{2}\left[E_{1}[u]\right]}_{L^{2}}\leq{}h^{2}\norm{u}_{L^{2}}.$$

\end{proof}

 \begin{prop}\label{prop:timeslice} Suppose $r\geq{} 3$. Under the assumptions of Theorem \ref{thm:localmain} and in the coordinate system developed in Proposition \ref{prop:factor}
  $$\left(\norm{u}_{L^{p}_{y}L^{2}_{t}L^{2}_{z}}+h^{-1}\norm{E_{1}[u]}_{L^{p}_{y}L^{2}_{t}L^{2}_{z}}\right)\lesssim \norm{u}_{L^{2}_{y}L^{2}_{t}L^{2}_{z}}.$$
 \end{prop}
 
 \begin{proof}
 We proceed as in the proof of Proposition \ref{prop:timeslice2} however this time we successively  use the evolution equations associated with $\tilde{p}_{i}^{(i+1,r)}(x,hD_{\tilde{x}^{(i+1,r)}})$ to estimate out the $L^{p}$ norm in each $y_{i}$. We begin with $\tilde{p}^{(3,r)}_{2}(x,\tilde{\xi}^{(3,r)})$ and factorise as
$$\tilde{p}^{(3,r)}_{2}(x,\tilde{\xi}^{(3,r)})=e_{2}(x,\tilde{\xi}^{(3,r)})(\xi_{2}-a_{2}(x,\tilde{\xi}^{(2,r)}))$$
with $|e_{2}(x,\tilde{\xi}^{(3,r)})|>c>0$. Therefore in $(t,y,z)$ coordinates we have
$$(hD_{y_{1}}-a_{2}(t,y,z,hD_{t},hD_{z}))u=E_{2}[u],\quad\norm{E_{2}[u]}_{L^2}\lesssim\norm{\tilde{p}^{(3,r)}_{2}(x,hD_{\tilde{x}^{(3,r)}})u}_{L^{2}}.$$
So as in Koch-Tataru-Zworski \cite{koch} we use Duhamel's principle. We can write
$$u=U_{2}(y_{1},0)u(t,0,\tilde{y}^{(1)},z)+\frac{i}{h}\int_{0}^{y_{1}}U_{2}(y_{1},s)E_{2}[u]ds$$
where
$$\begin{cases}
(hD_{y_{1}}-a_{2}(t,y,z,hD_{t},hD_{z}))U_{2}(y_{1},s)=0\\
U_{2}(s,s)=\Id.\end{cases}$$
Note that $a_{2}(t,y,z,hD_{t},hD_{z}))$ has no derivatives in $y_{2},\dots,y_{r-1}$ so again we treat these as parameters. Again the unitarity of $U_{2}$ gives us that
$$ \norm{u(\cdot,0,\tilde{y}^{(1)},\cdot)}_{L^{2}_{t}L^{2}_{z}}\lesssim\left(\norm{u(\cdot,\cdot,\tilde{y}^{(1)},\cdot)}_{L^{2}_{y_{1}}L^{2}_{t}L_{z}^{2}}+h^{-1}\norm{E_{2}[u](\cdot,\cdot,\tilde{y}^{(1)},\cdot)}_{L^{2}_{y_{1}}L^{2}_{t}L^{2}_{z}}\right).$$
So by shifting the zero in $y_{1}$ and the fact that we are in a compact set inside phase space
$$\norm{u(\cdot,\cdot,\tilde{y}^{(1)},\cdot)}_{L^{p}_{y_{1}}L^{2}_{t}L^{2}_{z}}\lesssim\left(\norm{u(\cdot,\cdot,\tilde{y}^{(1)},\cdot)}_{L^{2}_{y_{1}}L^{2}_{t}L_{z}^{2}}+h^{-1}\norm{E_{2}[u](\cdot,\cdot,\tilde{y}^{(1)},\cdot)}_{L^{2}_{y_{1}}L^{2}_{t}L^{2}_{z}}\right).$$
By treating $E_{1}[u]$ itself as a quasimode we also have
$$\norm{E_{1}[u](\cdot,\cdot,\tilde{y}^{(1)},\cdot)}_{L^{p}_{y_{1}}L^{2}_{t}L^{2}_{z}}\lesssim\left(\norm{E_{1}[u](\cdot,\cdot,\tilde{y}^{(1)},\cdot)}_{L^{2}_{y_{1}}L^{2}_{t}L_{z}^{2}}+h^{-1}\norm{E_{2}\left[E_{1}[u]\right](\cdot,\cdot,\tilde{y}^{(1)},\cdot)}_{L^{2}_{y_{1}}L^{2}_{t}L^{2}_{z}}\right).$$
So taking $L^{p}$ norms in $\tilde{y}^{(1)}$ we obtain
\begin{multline*}
\left(\norm{u}_{L^{p}_{y}L^{2}_{t}L^{2}_{z}}+h^{-1}\norm{E_{1}[u]}_{L^{p}_{y}L^{2}_{t}L^{2}_{z}}\right)\lesssim{}\Big(\norm{u}_{L^{p}_{\tilde{y}^{(1)}}L^{2}_{y_{1}}L^{2}_{t}L^{2}_{z}}+h^{-1}\norm{E_{1}[u]}_{L^{p}_{\tilde{y}^{(1)}}L^{2}_{y_{1}}L^{2}_{t}L^{2}_{z}}\\
+h^{-1}\norm{E_{2}[u]}_{L^{p}_{\tilde{y}^{(1)}}L^{2}_{y_{1}}L^{2}_{t}L^{2}_{z}}+h^{-2}\norm{E_{2}\left[E_{1}[u]\right]}_{L^{p}_{\tilde{y}^{(1)}}L^{2}_{y_{1}}L^{2}_{t}L^{2}_{z}}\Big).\end{multline*}
By repeating this process for
$$\tilde{p}^{(4,r)}_{3}(x,\tilde{\xi}^{(4,r)}),\dots,p_{r}(x,\xi)$$
we obtain
$$\left(\norm{u}_{L^{p}_{y}L^{2}_{t}L^{2}_{z}}+h^{-1}\norm{E_{1}[u]}_{L^{p}_{y}L^{2}_{t}L^{2}_{z}}\right)\lesssim \sum_{\alpha=0}^{r}\sum_{\boldsymbol{i}\in I_{\alpha}}h^{-\alpha}\norm{E_{i_{1}}\left[E_{i_{2}}\cdots\left[E_{i_{\alpha}}u\right]\right]}_{L^{2}}$$
where $I_{\alpha}=\{\boldsymbol{i}=(i_{1},\dots,i_{\alpha})\mid 1\leq{}i_{k}\leq{}r, i_{k+1}<i_{k}\}$. So we need only show that 
$$\norm{E_{i_{1}}\left[E_{i_{2}}\cdots\left[E_{i_{\alpha}}u\right]\right]}_{L^{2}}\lesssim h^{\alpha}\norm{u}_{L^{2}}.$$

The reasoning applied in the proof of Proposition \ref{prop:timeslice2} tells us that if $u$ is a strong joint order $h$ quasimode of $q(x,hD),p(x,hD)$ and $(hD_{x_{i}}-a(x,hD_{\tilde{x}^{(i)}}))$ it is also strong joint order $h$ quasimode of $q(x,hD)$ and $\tilde{p}^{(i)}(x,hD_{\tilde{x}^{(i)}})$ where
$$\tilde{p}^{(i)}(x,\tilde{\xi}^{(i)})=p(x,\xi_{1},\dots,\xi_{i-1},a(x,\tilde{\xi}^{(i)}),\xi_{i+1},\dots,\xi_{n}).$$
Applying this statement inductively following the process defined in the proof of Proposition \ref{prop:factor} we see that since $u$ is a strong joint order $h$ quasimode of $p_{i_{1}}(x,hD),\dots,p_{i_{\alpha}}(x,hD)$ 
$$\norm{\tilde{p}^{(i_{1}+1,r)}_{i_{1}}(x,hD_{\tilde{x}^{(i_{1}+1,r)}})\circ \cdots \circ \tilde{p}^{(i_{\alpha}+1,r)}_{i_{\alpha}}(x,hD)u}_{L^{2}}\lesssim{}h^{\alpha}\norm{u}_{L^{2}}.$$
On the other hand
$$E_{i}[v]=(hD_{x_{i}}-a_{i}(x,hD_{\tilde{x}^{(i,r)}}))v\quad i=1,\dots,r$$
where
$$\tilde{p}^{(i+1,r)}_{i}(x,\tilde{\xi}^{(i+1,r)})=e_{i}(x,\tilde{\xi}^{(i+1,r)})(\xi_{i}-a_{i}(x,\tilde{\xi}^{(i,r)}))$$
and $|e_{i}(x,\tilde{\xi}^{(i+1,r)})|>c>0$. So by the invertibility of the $e_{i}(x,hD_{\tilde{x}^{(i+1,r)}})$ we obtain
$$\norm{E_{i_{1}}\left[E_{i_{2}}\cdots\left[E_{i_{\alpha}}u\right]\right]}_{L^{2}}\lesssim h^{\alpha}\norm{u}_{L^{2}}.$$

\end{proof}

It only remains to treat the $n=r$ case. This is just an application of the same arguments as in  Propositions \ref{prop:timeslice2} and \ref{prop:timeslice}

\begin{thm}
\label{prop:localn=r}
Suppose $u$ is a strong joint order $h$ quasimodes of $p_{1}(x,hD),\dots,p_{n}(x.hD)$ satisfying the admissibility conditions of Theorem \ref{thm:n=r}. Let $\chi(x,\xi)$ be a smooth compactly supported functions localised near a point $(x_{0},\xi_{0})$ at which $p_{j}(x_{0},\xi_{0})=0$ for all $j$. Then
$$\norm{\chi(x,hD)u}_{L^{p}}\lesssim\norm{u}_{L^{2}}$$
for every $2\leq{}p\leq{}\infty$. 

\end{thm}

\begin{proof}
We run the same argument as in Proposition \ref{prop:timeslice} but begin with $\tilde{p}^{(2,n)}_{1}(x,\tilde{\xi}^{(2,n)})$. That is we factorise
$$\tilde{p}^{(2,n)}_{1}(x,\tilde{\xi}^{(2,n)})=e_{1}(x,\tilde{\xi}^{(2,n)})(\xi_{1}-a_{1}(x))$$
with $|e_{1}(x,\tilde{\xi}^{(2,n)})|\geq{}0$. Adopting the $(t,y)$ notation (since $r=n$ there is no $z$ coordinates) we write
$$u=e^{\frac{i}{h}\rho(t,0,y)}u(0,y)+\frac{i}{h}\int_{0}^{t}e^{\frac{i}{h}\rho(t,s,y)}E_{1}[u]ds$$
where
$$E_{1}[u]=(hD_{t}-a_{1}(t,y))u$$
and
$$\rho(t,s,y)=\int_{s}^{t}a_{1}(\tau,y)d\tau.$$
Then as in the proof of Proposition \ref{prop:timeslice}
$$\norm{u(\cdot,y)}_{L^{p}_{t}}\lesssim\left(\norm{u(\cdot,y)}_{L^{2}_{t}}+h^{-1}\norm{E_{1}[u](\cdot,y)}_{L^{2}_{t}}\right)$$
so
$$\norm{u}_{L^{p}}\lesssim\left(\norm{u}_{L^{p}_{y}L^{2}_{t}}+h^{-1}\norm{E_{1}[u]}_{L^{p}_{y}L^{2}_{t}}\right).$$
Now applying Proposition \ref{prop:timeslice}
$$\norm{u}_{L^{p}}\lesssim{}\norm{u}_{L^{2}}.$$
\end{proof}

\section*{Acknowledgements}
The author would like to acknowledge Suresh Eswarathasan for his suggestions on the exposition of this paper and Simon Marshall for explaining the connection to sub-convex bounds from number theory. The author would also like to acknowledge the reviewer for all their comments, particularly the suggestion to weaken the curvature condition from the original version of the paper. 
\bibliographystyle{plain}
\bibliography{references}
\end{document}